\documentclass[11pt, twoside]{article}
\usepackage{bbm}
\usepackage{amsmath,amsthm}
\usepackage{amssymb,latexsym}
\usepackage{mathrsfs}
\usepackage{enumerate}
\usepackage[numbers,sort&compress]{natbib}
\usepackage{indentfirst}

\newtheorem{theorem}{Theorem}[section]

\newtheorem{proposition}[theorem]{Proposition}

\theoremstyle{definition} \theoremstyle{remark}
\numberwithin{equation}{section}
\allowdisplaybreaks

\date{}

\begin{document}

%\markboth{\\ G. He and H. Zhang}{Uniform convergence of conditional distributions for one-dimensional diffusion processes }

%\date{November 24, 2016}

\title{{\bf Uniform convergence of conditional distributions for one-dimensional diffusion processes}}
\setcounter{footnote}{-1}
\author{Guoman He$^{1}$\thanks{\small\footnotesize{$^{1}$School of Mathematics and Statistics, Hunan University of Technology and Business, Changsha, Hunan 410205, PR China.
Email address: hgm0164@163.com}},~~~~~~
\setcounter{footnote}{-1}
Hanjun Zhang$^{2}$\thanks{\footnotesize {$^{2}$School of Mathematics and Computational Science, Xiangtan University, Xiangtan, Hunan 411105, PR China.
Email address: hjz001@xtu.edu.cn}}
\\[1.8mm]}

\baselineskip 0.23in

\maketitle

\begin{abstract}
In this paper, we study the quasi-stationary behavior of the one-dimensional diffusion process with a regular or exit boundary at 0 and an entrance boundary at $\infty$. By using the Doob's $h$-transform, we show that the conditional distribution of the process converges to its unique quasi-stationary distribution exponentially fast in the total variation norm, uniformly with respect to the initial distribution. Moreover, we also use the same method to show that the conditional distribution of the process converges exponentially fast in the $\psi$-norm to the unique quasi-stationary distribution. The rate of convergence of the conditional empirical measure to the quasi-ergodic distribution is also considered.
Finally, two examples arising in population dynamics are also given to illustrate the main results.
\\[2mm]
\\[1.2mm]
\end{abstract}

\baselineskip 0.234in
\section{Introduction}
\label{sect1}

Consider a one-dimensional diffusion process $X=(X_{t})_{t\geq0}$ on $[0,\infty)$ killed at 0, that is, for all $t\geq s$, if $X_{s}=0$, then $X_{t}=0$, given by the solution to the stochastic differential equation (SDE)
\begin{equation}
\label{1.1}
dX_t=dB_t-q(X_t)dt,~~~~~~~~~~X_{0}=x>0,
\end{equation}
where $(B_t)_{t\geq0}$ is a standard one-dimensional Brownian motion and the drift $q$ is continuously differentiable on $(0,\infty)$, that is, $q\in C^1((0,\infty))$.
\par
Let $T_0:=\inf\{t\geq0:X_t=0\}$ be the first hitting time of $0$. We write $\mathbb{P}_{x}$ for the probability measure of the process when $X_{0}=x$.
In this paper, we assume that absorption at 0 is {\em certain}, that is, for all $x>0$,
\begin{equation}
\label{1.2}
\mathbb{P}_x(T_{0}<\infty)=1.
\end{equation}
\par
The concept of the {\em stationarity} will no longer be appropriate to depict the asymptotic behavior of such a process conditioned on long-term survival. A natural concept will be the {\em quasi-stationarity}. The main concept of this theory is the {\em quasi-stationary distribution}, that is, a probability measure $\alpha$ on $(0,\infty)$ such that, for any $t\geq0$,
\begin{equation}
\label{1.3}
\mathbb{P}_{\alpha}(X_{t}\in\cdot|T_{0}>t)=\alpha,
\end{equation}
where $\mathbb{P}_{\alpha}(\cdot)=\int_{0}^{\infty}\mathbb{P}_{x}(\cdot)\alpha(dx)$. To be a quasi-stationary distribution, one possible approach is to look at the {\em quasi-limiting distribution} $\alpha$, that is, a probability measure $\alpha$ on $(0,\infty)$ such that there exists an initial distribution $\mu$ on $(0,\infty)$,
\begin{equation}
\label{1.4}
\lim\limits_{t\rightarrow\infty}\mathbb{P}_{\mu}(X_{t}\in\cdot|T_{0}>t)=\alpha.
\end{equation}
If (\ref{1.4}) holds, we also say that $\mu$ is attracted to $\alpha$, or we say that $\mu$ is in the domain of attraction of $\alpha$. For a general overview on quasi-stationarity, the readers are urged to refer to \cite{CMS13,MV12}.
\par
When studying quasi-stationary distributions of one-dimensional diffusion processes, the boundaries of an interval often need to be classified: {\em regular boundary}, {\em exit boundary}, {\em entrance
boundary} and {\em natural boundary}. To do so, let us introduce some notations. For $x\in(0,\infty)$, we define the {\em scale function} $\Lambda$ and the {\em speed measure} $m$ of the process $X$ respectively by
\begin{equation}
\label{1.5}
\Lambda(x)=\int_{1}^{x}e^{Q(y)}dy,~~~~~m(dy)=e^{-Q(y)}dy,
\end{equation}
where $Q(y)=\int_{1}^{y}2q(x)dx$.  Using the speed measure $m$ and the scale function $\Lambda$, the boundaries of $(0,\infty)$ can be
classified. Let $a=0$ or $\infty$ and fix a point $b\in(0,\infty)$. Set
\begin{equation}
\label{1.6}
I(a)=\int_{b}^{a}d\Lambda(y)\int_{b}^{y}dm(z),~~J(a)=\int_{b}^{a}dm(y)\int_{b}^{y}d\Lambda(z).
\end{equation}
According to \cite{I06}, the boundary $a$ can be classified as follows:
\begin{equation*}
   {\rm The ~boundary}~ a ~{\rm is~ said~ to~ be}~\left\{\begin{array}{ll} {\rm regular} &{\rm if}\ I(a)<\infty ~~{\rm and}~~ J(a)<\infty,\\ {\rm exit} &{\rm if}\ I(a)<\infty~~{\rm and}~~J(a)=\infty,\\  {\rm entrance} &{\rm if}\ I(a)=\infty~~{\rm and}~~J(a)<\infty,\\ {\rm natural} &{\rm if}\ I(a)=\infty~~{\rm and}~~J(a)=\infty. \end{array}\right.
\end{equation*}
\par
When $0$ is an exit boundary, the speed measure $m$ is an infinite measure. At this time, it is a great challenge to prove the existence of a quasi-stationary distribution, due to it is difficult to guarantee $\eta_1\in\mathbb{L}^1(m)$, which is the basic condition for constructing a quasi-stationary distribution. Here, $\eta_1$ is the eigenfunction corresponding to the first non-trivial eigenvalue of the operator of the process $X$. In this case, Cattiaux et~al.\cite{CCLMMS09} had proved the existence and uniqueness of quasi-stationary distributions for the process $X$ under a set of assumptions which are mainly to ensure that the spectrum of the infinitesimal operator is discrete and $\eta_1\in\mathbb{L}^1(m)$.
In 2012, Littin \cite{LJ12} gave a relaxed condition. Under the sole condition that $\infty$ is an entrance boundary, He proved the existence and uniqueness of a quasi-stationary distribution and that of the Yaglom limit for the process $X$. Similarly, under the same condition that $\infty$ is an entrance boundary, Hening and Kolb \cite{HK19} also demonstrated the existence and uniqueness of the quasi-stationary distribution and showed that this distribution attracts all initial distributions with support on $(0,\infty)$ in 2019. In 2024, Shen et~al.\cite{SWY24} proved the concentration and tightness of quasi-stationary distributions for small-noise one-dimensional diffusions, showing that the weak limits of these distributions are invariant measures supported on the attractor of the deterministic system.
When $0$ is a regular boundary, the speed measure $m$ is a finite measure, and then it is easy to get $\eta_1\in\mathbb{L}^1(m)$. In this case, many scholars study quasi-stationary distributions of one-dimensional diffusion processes with a natural or entrance boundary at $\infty$. For example, Mart\'{\i}nez et~al.\cite{MPM98} gave an optimal sufficient condition for an initial measure to be in the domain of attraction of a quasi-stationary distribution for the Brownian
motion with strictly negative constant drift. Lladser and San Mart\'{\i}n \cite{LM00} studied the domain of attraction of continuum family of quasi-stationary distributions for the Ornstein-Uhlenbeck process. Kolb and Steinsaltz \cite{KS12} gave some criteria for determining whether
one-dimensional diffusions with killing converges to a quasi-stationary distribution for all cases in which the bottom of the spectrum does not coincide with the limit of the killing rate at $\infty$. Zhang and He \cite{ZH16} studied the existence, uniqueness and domain of attraction of quasi-stationary distributions for one-dimensional diffusions when $0$ is a regular boundary and $\infty$ is an entrance boundary.
O\c{c}afrain \cite{O20} studied the polynomial rate of convergence to the Yaglom limit of
Brownian motion with strictly negative constant drift in 1-Wasserstein distance. Yamato \cite{YK22,YK24} studied the non-minimal quasi-stationary distributions of one-dimensional diffusions via the first hitting uniqueness and the renewal dynamical approach. Fang and Mao \cite{FM26} studied the rate of convergence to the Yaglom limit for the radial Ornstein-Uhlenbeck processes with negative drift in the weighted total variation norm.
\par
This paper is closely related to \cite{CCLMMS09, HK19, KS12, LJ12, ZH16}. These papers only consider the existence, uniqueness and domain of attraction of quasi-stationary distributions. However, this does not imply that (\ref{1.4}) is uniform convergence. The reverse is true. This paper, we are interested in the speed of convergence of (\ref{1.4}). More formally, we are looking forward to the existence of a quasi-stationary distribution $\alpha$ on $(0,\infty)$ such that,
for all probability measure $\mu$ on $(0,\infty)$ and all $t\geq0$, there exist two positive constants $C, \gamma$ such that
\begin{equation}
\label{1.20}
\left\|\mathbb{P}_{\mu}(X_t\in \cdot|T_{0}>t)-\alpha\right\|_{TV}\leq Ce^{-\gamma t}.
\end{equation}
Here, $\|\cdot\|_{TV}$ is the {\em total variation norm}, defined by
\begin{equation*}
\|\nu-\mu\|_{TV}:=\sup\limits_{g\in\mathcal{B}_{1}(0,\infty)}|\nu(g)-\mu(g)|,
\end{equation*}
where $\nu,\mu$ are any two probability measures on $(0,\infty)$, $\mu(g)=\int_{0}^{\infty}g(x)\mu(dx)$ and $\mathcal {B}_{1}(0,\infty)$ is the set of the measurable bounded functions defined on $(0,\infty)$ such that $\|g\|_{\infty}\leq1$. Here, $\|g\|_{\infty}=\sup\limits_{x\in(0,\infty)}|g(x)|$. The total variation norm has been used by many scholars to quantify (\ref{1.4}) (see, e.g., \cite{CV16, CV17, CV18, HZ22, O21}). For general Markov processes, the exponential convergence (\ref{1.20}) contains many interesting properties. One is that it implies the process admits a unique quasi-stationary distribution and all initial distributions are in the domain of attraction of this unique quasi-stationary distribution. Another is that it implies the $Q$-process (the process conditioned to never be absorbed) is exponential ergodic. See \cite{CV16, CV17b, HZY21} for other more detailed results. For the one-dimensional diffusion process $X$, under our conditions, we point out that we can use another method to directly prove that the $Q$-process is strongly ergodic. For a one-dimensional diffusion (with killing) in natural scale, the uniform exponential convergence (\ref{1.20}) has already been proved by Champagnat and Villemonais \cite{CV17,CV18} through probabilistic methods.
\par
In this paper, we study the uniform exponential convergence (\ref{1.20}) for the one-dimensional diffusion process $X$ with a regular or exit boundary at $0$ and an entrance boundary at $\infty$. When $0$ is a regular boundary and $\infty$ is an entrance boundary, the uniform exponential convergence in total variation has already been established in \cite[Theorem 4.1 and Corollary 4.1]{CV18} for one-dimensional diffusions in natural scale, but here we use a new proof method. When $0$ is an exit boundary and $\infty$ is an entrance boundary, the results of \cite{CV18} provide sufficient conditions for uniform exponential convergence under certain regularity assumptions on the speed measure near $0$ (see \cite[Theorem 4.3]{CV18} and the discussion in Sections 4.2--4.3 therein). However, as explicitly pointed out in \cite[Section 4.3]{CV18}, their framework does not cover the case where the speed measure exhibits strong oscillations near the boundary $0$, which is left as an open problem.
The main contribution of the present paper is to fill this gap. By employing Doob's $h$-transform (or $h$-transform), we establish the uniform exponential convergence for the exit-boundary case without any oscillation condition on the speed measure near $0$, thereby extending the results of \cite{CV18} to the previously untreated oscillatory setting. Moreover, we also use this method to show that the conditional distribution of the process converges exponentially fast in the $\psi$-norm to the unique quasi-stationary distribution. This method differs fundamentally from the probabilistic approach in \cite{CV17,CV18} and has been successfully applied to the birth and death process by the authors \cite{HZ22}.
\par
The following result is one of our main results.
\begin{theorem}
\label{thm2.1}
Let $X$ be a one-dimensional diffusion process satisfying $(\ref{1.2})$. Assume that $0$ is a regular or exit boundary. Then, the following are equivalent$:$
\par $(\mathrm{i})$~$\infty$ is an entrance boundary.
\par $(\mathrm{ii})$~There is precisely one quasi-stationary distribution.
\par $(\mathrm{iii})$~There exist a probability measure $\alpha$ on $(0,\infty)$ and two positive constants $C, \gamma$ such that, for all initial distributions $\mu$ on $(0,\infty)$ and all $t\geq0$,
\begin{equation*}
\left\|\mathbb{P}_{\mu}(X_t\in \cdot|T_{0}>t)-\alpha\right\|_{TV}\leq Ce^{-\gamma t}.
\end{equation*}
Moreover, the distribution $\alpha$ in $(\mathrm{iii})$ is the unique quasi-stationary distribution for the process $X$.
\end{theorem}
\par
For any measurable function $\psi:(0,\infty)\rightarrow[1,+\infty)$, we define the $\psi$-norm of a signed measure $\nu$ by $\|\nu\|_{\psi}=\sup\limits_{|g|\leq\psi}|\nu(g)|$. We can see that when $\psi=1$, it is nothing but the total variation norm. We denote by $\|\cdot\|_{2}$ the $\mathbb{L}^2(\beta)$-norm, defined as $\|g\|_{2}=(\int_{0}^{\infty}g^{2}(x)\beta(dx))^{\frac{1}{2}}$, where $\beta$ is the unique invariant probability measure of the $Q$-process $Y$ defined in Section \ref{sect2}.
For any positive measure $\mu$ on $(0,\infty)$ and any measurable function $g$ on $(0,\infty)$ satisfying $\mu(g)<\infty$, we denote by $g\circ\mu$ the probability measure defined as
\begin{equation}
\label{1.7}
g\circ\mu(dx):=\frac{g(x)\mu(dx)}{\mu(g)}.
\end{equation}
We denote by $\mathbb{E}_{\mu}$ and $\mathbb{E}_{x}$ the expectation with respect to $\mathbb{P}_{\mu}$ and $\mathbb{P}_{x}$ respectively. For any two probability measures $\mu, \nu$ on $(0,\infty)$, $\mu\ll\nu$ denotes $\mu$ is absolutely continuous with respect to $\nu$. For any initial distribution $\mu$ on $(0,\infty)$, by (\ref{2.2}), we get that $\mu\ll\alpha$, where $\alpha$ is the unique quasi-stationary distribution of the  process $X$.
\par
The second main result of this paper is the following stronger convergence result.
\begin{theorem}
\label{thm2.2}
Let $X$ be a one-dimensional diffusion process satisfying $(\ref{1.2})$. Assume that $0$ is a regular or exit boundary, $\infty$ is an entrance boundary and there exists a function $\psi:(0,\infty)\rightarrow[1,+\infty)$ such that $\alpha(\psi)<+\infty$ and $\alpha(\psi^{2}/\eta_1)<+\infty$, where $\alpha$ is the unique quasi-stationary distribution of the process $X$. Then, for any initial distribution $\mu$ on $(0,\infty)$, there exist $t_{\mu}$ and $\gamma>0$ such that, for any $t\geq t_{\mu}$,
\begin{equation*}
\sup\limits_{|g|\leq\psi}|\mathbb{E}_{\mu}[g(X_{t})|T_{0}>t]-\alpha(g)|\leq \max\{D_{1}, D_{2}\}\left[\frac{\alpha(\frac{\psi^2}{\eta_1})}{m(\eta_{1})}\right]^{\frac{1}{2}}\|\frac{d(\eta_1\circ\mu)}{d(\eta_1\circ\alpha)}-1\|_{2}e^{-\gamma t},
\end{equation*}
where
\begin{eqnarray*}
D_{1}=\left(1+\frac{1+\alpha(\psi)}{1-c}\right),~~~D_{2}=2+\alpha(\psi)~~~\mathrm{and}~~~c\in(0,1).
\end{eqnarray*}
\end{theorem}
As an application of the main results, we obtain the following result. This result implies that when $t$ goes to infinity, the polynomial convergence of conditional distributions $\frac{1}{t}\int_{0}^{t}\mathbb{P}_{\mu}(X_s\in \cdot|T_{0}>t)ds$ toward the quasi-ergodic distribution $\beta$, which improves the result obtained by He \cite[Theorem 3.1]{H18} who only gave a weak convergence result. A probability measure $\beta$ on $(0,\infty)$ is said to be a {\em quasi-ergodic distribution}, if for all $x\in(0,\infty)$, $t>0$ and all bounded measurable function $g$ on $(0,\infty)$,
\begin{equation*}
\lim_{t\rightarrow\infty}\mathbb{E}_{x}\left(\frac{1}{t}\int_{0}^{t}g(X_s)ds|T_{0}>t\right)=\beta(g).
\end{equation*}
There is an essential differences between a quasi-stationary distribution and a quasi-ergodic distribution. For work on this topic,
we refer the reader to \cite{CV17b,H18, HZ16, HZY21} and the references therein.
\begin{proposition}
\label{pro1.333}
Let $X$ be a one-dimensional diffusion process satisfying $(\ref{1.2})$. Assume that $0$ is a regular or exit boundary and $\infty$ is an entrance boundary. Then, there exists a positive constant $G$ such that, for all $x\in(0,\infty)$, $t>0$ and all bounded measurable function $g$ on $(0,\infty)$,
\begin{equation*}
\left|\mathbb{E}_{x}\left(\frac{1}{t}\int_{0}^{t}g(X_s)ds|T_{0}>t\right)-\beta(g)\right|\leq \frac{G\|g\|_{\infty}}{t},
\end{equation*}
where $\beta$ is as in $(\ref{2.4})$.
\end{proposition}
The rest of this paper is arranged as follows. In Section \ref{sect2}, we present some preliminaries that will be needed in
the proof of our main results. In Section \ref{sect3}, we are devoted to giving the proof of Theorem \ref{thm2.1}. We are devoted to giving the proof of Theorem \ref{thm2.2} and Proposition \ref{pro1.333} in Section \ref{sect4} and Section \ref{sect5} respectively. Two examples are given to study in Section \ref{sect6}.

\section{Preliminaries}
\label{sect2}
In this section, we mainly present some known results on the spectrum of the generator and quasi-stationary distributions of the process $X$ and discuss some properties of the $Q$-process. Associated to the process $X$, we consider the sub-Markovian semigroup $(P_{t})_{t\geq0}$, that is, $0\leq P_{t}g\leq1$ $m\text{-}a.e.$ if $0\leq g\leq1$, given by
\begin{equation*}
P_{t}g(x)=\mathbb{E}_x[g(X_t){\bf{1}}_{\{T_{0}>t\}}].
\end{equation*}
The generator of this semigroup is given by $\mathcal {L}=\frac{1}{2}\frac{d^2}{d x^2}-q(x)\frac{d}{d x}$.
\par
To ensure the existence of a quasi-stationary distribution, the bottom of the spectrum of the generator is often necessary to be strictly positive and the lowest eigenfunction is integrable with respect to the reference measure. On the spectrum of $\mathcal {L}$ and the quasi-stationary distribution for the process $X$, we summarize the results proved in \cite{HK19, LJ12, KS12, ZH16} as follows. If $0$ is a regular boundary and $\infty$ is an entrance boundary, then it can be found in \cite{KS12} and \cite{ZH16}. When $0$ is an exit boundary and $\infty$ is an entrance boundary, it mainly comes from \cite{HK19} and \cite{LJ12}.
\begin{proposition}
\label{prop 2.1}
For the one-dimensional diffusion process $X$ satisfying $(\ref{1.2})$, if $0$ is a regular or exit boundary and $\infty$ is an entrance boundary, then
\par $(\mathrm{i})$ $-\mathcal {L}$ has a purely discrete spectrum. The eigenvalues $0<\lambda_1<\lambda_2<\cdots$ are simple, $\lim_{n\rightarrow\infty}\lambda_n=+\infty$, and the eigenfunction $\eta_n$ associated to $\lambda_n$ has exactly $n$ roots belonging to $(0,\infty)$. The eigenfunction sequence $(\eta_n)_{n\geq1}$ is an orthonormal basis of $\mathbb{L}^2(m)$ and $\eta_1$ can be taken to be strictly positive on $(0,\infty)$.
\par $(\mathrm{ii})$ for any $n\geq1, \eta_n\in \mathbb{L}^1(m)$.
\par $(\mathrm{iii})$ for all $x > 0$, $t > 0$ and all bounded measurable function $g$ on $(0,\infty)$, there exists some
density $r(t,x,\cdot)$ satisfying
\begin{equation}
\label{2.1}
\mathbb{E}_x[g(X_t){\bf{1}}_{\{T_{0}>t\}}]=\int_{0}^{\infty}r(t,x,y)g(y)m(dy).
\end{equation}
Moreover, for all $x>0$ and $t>0$, $r(t,x,\cdot)\in\mathbb{L}^2(m)$.
\par $(\mathrm{iv})$ there is precisely one quasi-stationary distribution $\alpha$ for the process $X$, given by
\begin{equation}
\label{2.2}
\alpha(dx)=\frac{\eta_{1}(x)m(dx)}{m(\eta_{1})}.
\end{equation}
Moreover, $\alpha$ attracts all initial distributions $\mu$ on $(0,\infty)$. Also, for any $x\in(0,\infty)$ and any Borel subset $B\subseteq(0,\infty)$,
\begin{equation*}
\lim\limits_{t\rightarrow\infty}e^{\lambda_{1}t}\mathbb{P}_{x}(T_{0}>t)=\eta_{1}(x)m(\eta_{1}),
\end{equation*}
\begin{equation*}
\lim\limits_{t\rightarrow\infty}e^{\lambda_{1}t}\mathbb{P}_{x}(X_{t}\in B, T_{0}>t)=\alpha(B)m(\eta_{1}).
\end{equation*}
\end{proposition}
Moreover, for the one-dimensional diffusion process $X$ satisfying $(\ref{1.2})$, if $0$ is a regular or exit boundary and $\infty$ is an entrance boundary, from \cite[Theorem 5.4]{TM19} and the proof of \cite[Lemma 4.4]{KS12} and \cite[Theorem 2.8]{HK19}, we know that $\eta_1$ can be chosen such that $\eta_1\in C^2((0,\infty))$ and satisfies $\frac{1}{2}\eta''_1-q\eta'_1=-\lambda_1 \eta_1$.
\par
Now, we introduce the $Q$-process, which is defined as a Doob's $h$-transform of the semigroup $(P_{t})_{t\geq0}$. Here, we take $h=\eta_1$. In other words, we describe the law of the process $X$ conditioned to never be absorbed, usually called the $Q$-process. We denote by $Y=(Y_{t})_{t\geq0}$ the $Q$-process. For the process $X$ satisfying $(\ref{1.2})$, if $0$ is a regular or exit boundary and $\infty$ is an entrance boundary, then we can use the same proof as in \cite[Corollary 6.1]{CCLMMS09} to show that the $Q$-process $Y$ exists. In fact, the key elements of the proof of \cite[Corollary 6.1]{CCLMMS09} only need to know that $-\mathcal {L}$ has a purely discrete spectrum, $\eta_1\in \mathbb{L}^1(m)$ and the last part of Proposition \ref{prop 2.1} holds. More precisely, if $0$ is a regular or exit boundary and $\infty$ is an entrance boundary, then for all $x>0$, $s\geq0$ and all $A$ Borel measurable subsets of $C([0, s])$,
\begin{equation*}
\mathbb{Q}_{x}(Y\in A)=\lim_{t\to\infty}\mathbb{P}_{x}(X\in A|T_{0}>t)
\end{equation*}
is well-defined, and the process $Y$ is a diffusion process on $(0,\infty)$ with transition probability densities (with respect to the Lebesgue measure) given by
\begin{equation}
\label{2.222}
\widetilde{r}(s,x,y)=e^{\lambda_{1}s}\frac{\eta_{1}(y)}{\eta_{1}(x)}r(s,x,y)e^{-Q(y)},
\end{equation}
that is, $\mathbb{Q}_{x}$ is locally absolutely continuous with respect to $\mathbb{P}_{x}$ and
\begin{equation}
\label{2.3}
\mathbb{Q}_{x}(Y\in A)=\mathbb{E}_{x}\left({\bf{1}}_{A}(X)e^{\lambda_{1}s}\frac{\eta_{1}(X_{s})}{\eta_{1}(x)}, T_{0}>s\right).
\end{equation}
The equality (\ref{2.3}) implies that
\begin{equation}
\label{2.5}
\widetilde{P}_{t}g(x)=\frac{e^{\lambda_{1}t}}{\eta_{1}(x)}P_{t}(\eta_{1}g)(x),
\end{equation}
where $(\widetilde{P}_{t})_{t\geq0}$ denotes the semigroup of $Y$. From (\ref{2.5}), we obtain
\begin{equation}
\label{2.6}
P_{t}g(x)=\eta_{1}(x)e^{-\lambda_{1}t}\widetilde{P}_{t}(\frac{g}{\eta_{1}})(x).
\end{equation}
So, it can be seen that the process $Y$ is a Doob's $h$-transform of the process $X$. The Doob's $h$-transform has a lot of nice properties. First, (\ref{2.6}) naturally implies that the quasi-stationarity of the process $X$ can be studied by using the Doob's $h$-transform. In fact, this method has been used by many scholars successfully to discuss the quasi-stationarity for some classes of Markov processes (see, e.g., \cite{DL15, HZ22, O20,O21,T19}). Besides, another useful piece of information is that the spectrum is invariant under Doob's $h$-transform  (see \cite[Chapter 4, Sections 3 and 10]{PRG95}).
\par
Under our assumptions, the same proof as in \cite[Corollary 6.2]{CCLMMS09} works, we know that the process $Y$ is an ergodic diffusion process on $(0,\infty)$ and admits a unique invariant probability measure
\begin{equation}
\label{2.4}
\beta(dx)=\eta^2_{1}(x)m(dx).
\end{equation}
In fact, we can show that the process $Y$ is strongly ergodic, that is, there exists a constant $\gamma>0$ such that
\begin{equation*}
\lim\limits_{t\rightarrow\infty}e^{\gamma t}\sup\limits_{x\in(0,\infty)}\left\|\widetilde{P}_{t}(x,\cdot)-\beta\right\|_{TV}=0,
\end{equation*}
where $\widetilde{P}_{t}(x,dy)$ is the transition function of the process $Y$.
\begin{proposition}
\label{prop 2.2}
For the one-dimensional diffusion process $X$ satisfying $(\ref{1.2})$, if $0$ is a regular or exit boundary and $\infty$ is an entrance boundary, then the process $Y$ is strongly ergodic.
\end{proposition}
\begin{proof}
First, note that the process $Y$ is the unique solution to the following SDE
\begin{equation}
\label{2.7}
dY_t=dB_t-\left(q(Y_t)-\frac{\eta'_{1}(Y_t)}{\eta_{1}(Y_t)}\right)dt,~~~~~~~~~~Y_{0}=y>0.
\end{equation}
The process $Y$ never reaches $0$ because its drift is of order $1/y$ for $y$ near $0$; see the discussion in \cite[after (3)]{MM04}.
This means that 0 is an inaccessible boundary. We define $\widetilde{Q}(y)$ analogous to $Q(y)$:
\begin{equation*}
\widetilde{Q}(y)=\int_{1}^{y}2\widetilde{q}(x)dx=\int_{1}^{y}2\left(q(x)-\frac{\eta'_{1}(x)}{\eta_{1}(x)}\right)dx=Q(y)-2\log\frac{\eta_{1}(y)}{\eta_{1}(1)}.
\end{equation*}
According to \cite{LJ12}, $\eta_{1}(x)$ is an increasing function on $(0,\infty)$. Also, we know from \cite[Theorem 5.4]{TM19} or \cite[Proposition 2.3]{HZ16} that $\eta_{1}(x)$ is bounded on $(0,\infty)$. So, for $b\in(0,\infty)$, we have
\begin{eqnarray}
\label{2.8}
\widetilde{I}(\infty)&:=&\int_{0}^{\infty}d\widetilde{\Lambda}(y)\int_{0}^{y}d\widetilde{m}(z)\nonumber\\
                     &\geq&\int_{b}^{\infty}d\widetilde{\Lambda}(y)\int_{b}^{y}d\widetilde{m}(z)\nonumber\\
                     &=&\int_{b}^{\infty}e^{\widetilde{Q}(y)}\left(\int_{b}^{y}e^{-\widetilde{Q}(z)}dz\right)dy\nonumber\\
                     &=&\int_{b}^{\infty}\frac{e^{Q(y)}}{\eta^{2}_{1}(y)}\left(\int_{b}^{y}e^{-Q(z)}\eta^{2}_{1}(z)dz\right)dy\nonumber\\
                     &\geq&\int_{b}^{\infty}e^{Q(y)}\left(\int_{b}^{y}e^{-Q(z)}dz\right)dy\nonumber\\
                     &=&I(\infty)\nonumber\\
                     &=&\infty
\end{eqnarray}
and
\begin{equation}
\label{2.9}
\int_{0}^{\infty}e^{-\widetilde{Q}(y)}dy=\frac{1}{\eta^{2}_{1}(1)}<+\infty.
\end{equation}
By (\ref{2.8}) and (\ref{2.9}), we know from \cite[Chapter 5]{C05} that the process $Y$ is ordinary ergodic and admits a stationary distribution $\beta(dx)=\frac{e^{-\widetilde{Q}(x)}dx}{\int_{0}^{\infty}e^{-\widetilde{Q}(x)}dx}=\eta^2_{1}(x)m(dx)$.
\par
For the process $Y$, we can conclude that $\infty$ is an entrance boundary. In fact, we have
\begin{eqnarray}
\label{2.11}
\widetilde{J}(\infty)&:=&\int_{b}^{\infty}d\widetilde{m}(y)\int_{b}^{y}d\widetilde{\Lambda}(z)\nonumber\\
                     &=&\int_{b}^{\infty}e^{-\widetilde{Q}(y)}\left(\int_{b}^{y}e^{\widetilde{Q}(z)}dz\right)dy\nonumber\\
                     &=&\int_{b}^{\infty}e^{\widetilde{Q}(y)}\left(\int_{y}^{\infty}e^{-\widetilde{Q}(z)}dz\right)dy\nonumber\\
                     &=&\int_{b}^{\infty}\frac{e^{Q(y)}}{\eta^{2}_{1}(y)}\left(\int_{y}^{\infty}e^{-Q(z)}\eta^{2}_{1}(z)dz\right)dy\nonumber\\
                     &\leq&\frac{\|\eta_{1}\|_{\infty}^{2}}{\eta^{2}_{1}(b)}\int_{b}^{\infty}e^{Q(y)}\left(\int_{y}^{\infty}e^{-Q(z)}dz\right)dy\nonumber\\
                     &=&\frac{\|\eta_{1}\|_{\infty}^{2}}{\eta^{2}_{1}(b)}{J}(\infty)\nonumber\\
                     &<&\infty.
\end{eqnarray}
Hence, by (\ref{2.8}) and (\ref{2.11}), we can claim that $\infty$ is an entrance boundary. For the process $Y$, we can also conclude that $0$ is an entrance boundary. In fact, fix a point $b\in(0,\infty)$, we can define $\widetilde{I}(0)$ and $\widetilde{J}(0)$ similar to ${I}(0)$ and ${J}(0)$ as follows.
\begin{equation*}
\widetilde{I}(0)=\int_{0}^{b}e^{\widetilde{Q}(y)}\left(\int_{y}^{b}e^{-\widetilde{Q}(z)}dz\right)dy,~~\widetilde{J}(0)=\int_{0}^{b}e^{-\widetilde{Q}(y)}\left(\int_{y}^{b}e^{\widetilde{Q}(z)}dz\right)dy.
\end{equation*}
We recall that 0 is an inaccessible boundary, then $0$ can only be either a natural boundary or an entrance boundary. If $0$ is an entrance boundary, then $\widetilde{I}(0)=\infty$ and $\widetilde{J}(0)<\infty$. In this case, we have $\int_{0}^{b}e^{\widetilde{Q}(y)}dy=\infty$ and $\int_{0}^{b}e^{-\widetilde{Q}(y)}dy<\infty$. If $0$ is a natural boundary, then $\widetilde{I}(0)=\infty$ and $\widetilde{J}(0)=\infty$. In this case, we have $\int_{0}^{b}e^{\widetilde{Q}(y)}dy=\infty$ and $\int_{0}^{b}e^{-\widetilde{Q}(y)}dy=\infty$. But, by (\ref{2.9}), we know that $\int_{0}^{b}e^{-\widetilde{Q}(y)}dy\leq\int_{0}^{\infty}e^{-\widetilde{Q}(y)}dy<\infty$. So, $0$ is an entrance boundary.
\par
We now prove that the process $Y$ is strongly ergodic. Our method is based on the Foster-Lyapunov method for strong ergodicity
of Markov processes, developed in \cite{DMT95}. Let
\begin{equation*}
\bar{L}(x):=1+\int_{b}^{x}e^{\widetilde{Q}(y)}\left(\int_{y}^{\infty}e^{-\widetilde{Q}(z)}dz\right)dy,~~~x\geq b.
\end{equation*}
So, we have
\begin{equation*}
\bar{L}(\infty)=1+\int_{b}^{\infty}e^{\widetilde{Q}(y)}\left(\int_{y}^{\infty}e^{-\widetilde{Q}(z)}dz\right)dy=1+\widetilde{J}(\infty)\leq1+\frac{\|\eta_{1}\|_{\infty}^{2}}{\eta^{2}_{1}(b)}{J}(\infty):=M<\infty.
\end{equation*}
Hence, for all $x\ge b$,
\begin{equation*}
1\le \bar{L}(x)\le M.
\end{equation*}
Now extend $\bar{L}(x)$ to the entire interval $(0,\infty)$ by setting
\begin{equation}
L(x):=
\begin{cases}
\bar L(b)=1, & 0<x\leq b,\\
\bar L(x), & x\geq b.
\end{cases}
\end{equation}
Thus,
\begin{equation}
\label{2.13}
1\le L(x)\le M<\infty,\qquad \forall x>0,
\end{equation}
and $L(x)\in C^2((0,\infty))$.
\par
For $x\geq b$, direct differentiation gives
\begin{equation*}
L'(x)=e^{\widetilde{Q}(x)}\left(\int_{x}^{\infty}e^{-\widetilde{Q}(z)}dz\right),
\end{equation*}
and
\begin{equation*}
L''(x)=2\widetilde{q}(x)L'(x)-1.
\end{equation*}
Let $\widetilde{\mathcal {L}}$ denotes the generator of the semigroup $(\widetilde{P}_{t})_{t\geq0}$ of the process $Y$. Applying the generator $\widetilde{\mathcal {L}}$ to $L$ yields
\begin{equation*}
\widetilde{\mathcal {L}} L(x)=\frac{1}{2} L''(x)-\widetilde{q}(x)L'(x)=-\frac{1}{2}.
\end{equation*}
Therefore, by (\ref{2.13}), for $x\geq b$, we have
\begin{equation}
\widetilde{\mathcal {L}}L(x)\leq -\frac{1}{2M}L(x).
\end{equation}
For $0<x\leq b$, $L(x)\equiv1$, so $\widetilde{\mathcal {L}}L(x)=0$.
Define
\begin{equation*}
c_{1}:=\frac{1}{2M},\qquad \hat{C}:=(0,b],
\qquad
b_0:=\frac{1}{2M}\sup_{0<x\leq b}L(x)=\frac{1}{2M}.
\end{equation*}
Then the following drift condition holds on the entire state space $(0,\infty)$:
\begin{equation}
\label{2.15}
\widetilde{\mathcal {L}} L(x)\leq -c_{1}L(x)+b_0\mathbf 1_{\hat{C}}(x),\qquad x>0.
\end{equation}
\par
Since $0$ is an entrance boundary, the process $Y$ never reaches $0$. For any $\varepsilon\in(0,b)$, we restrict the state space to
\begin{equation*}
E_\varepsilon:=[\varepsilon,\infty).
\end{equation*}
It is well-known that the compact set $C_\varepsilon:=[\varepsilon,b]$ is petite and the process $Y$ is $\varphi$-irreducible and aperiodic \cite{T19}. In addition, it is easy to see that the Lyapunov function $L(x)$ is bounded on $E_\varepsilon$ and satisfies $L(x)\geq1$, and the drift condition (\ref{2.15}) holds on $E_\varepsilon$:
\begin{equation}
\widetilde{\mathcal {L}} L(x)\leq -c_{1}L(x)+b_0\mathbf 1_{C_\varepsilon}(x).
\end{equation}
Thus, all the hypotheses of \cite[Theorem 5.2]{DMT95} are satisfied. Consequently, there exist constants $D_\varepsilon>0$ and $\gamma>0$ such that
\begin{equation}
\sup_{x\in E_\varepsilon}
\left\|\widetilde{P}_{t}(x,\cdot)-\beta\right\|_{TV}\leq D_\varepsilon L(x)e^{-\gamma t}\leq D_\varepsilon M e^{-\gamma t}.
\end{equation}
Since $\varepsilon>0$ is arbitrary and the process never enters $(0,\varepsilon)$, the above estimate holds uniformly on $(0,\infty)$:
\begin{equation}
\sup_{x>0}\left\|\widetilde{P}_{t}(x,\cdot)-\beta\right\|_{TV}\leq D M e^{-\gamma t},\qquad t\ge0.
\end{equation}
Therefore, the $Q$-process $Y$ is strongly ergodic. This completes the proof of Proposition \ref{prop 2.2}.
\end{proof}
\par
In this paper, as mentioned above, it is natural for us to use the Doob's $h$-transform to study the quasi-stationarity of one-dimensional diffusion processes. We point out that Proposition \ref{prop 2.2} plays an important key role in the proof of our main results. It also has an independent interest. Cattiaux et~al.\cite[Corollary 6.2]{CCLMMS09} just proved that the process $Y$ is ordinary ergodic.

\section{Proof of Theorem \ref{thm2.1}}
\label{sect3}
If $0$ is a regular boundary, it's a well-known fact that $(i)$ and $(ii)$ are equivalent (see \cite[Theorem 4.14]{KS12}), and if $0$ is an exit boundary, $(i)$ and $(ii)$ are also equivalent (see \cite[Corollary 1.14]{HK19}). Now let's turn our attention to prove that $(i)$ and $(iii)$ are equivalent. If $(iii)$ holds, then we know from \cite[Theorem 2.1]{CV16} that there is precisely one quasi-stationary distribution for $X$, and that the distribution $\alpha$ defined in (\ref{2.2}) is the unique quasi-stationary distribution. Therefore, $(i)$ holds.
\par
Conversely, if $(i)$ holds, first note that the measure $\beta$ is a reversible measure for the semigroup $(\widetilde{P}_{t})_{t\geq0}$. In fact, if $f,g\in\mathbb{L}^2(m)$, so that $f,g\in\mathbb{L}^2(\beta)$, then we have
\begin{eqnarray*}
\int_{0}^{\infty}(\widetilde{P}_{t}f)gd\beta&=&\int_{0}^{\infty}\eta^{2}_1(\frac{e^{\lambda_{1}t}}{\eta_1}{P}_{t}\eta_1f)gdm\\
                                            &=&\int_{0}^{\infty}\eta^{2}_1f(\frac{e^{\lambda_{1}t}}{\eta_1}{P}_{t}\eta_1g)dm\\
                                            &=&\int_{0}^{\infty}f(\widetilde{P}_{t}g)d\beta.
\end{eqnarray*}
Moreover, by Proposition \ref{prop 2.2}, we know that the $Q$-process $Y$ is strongly ergodic. For the $Q$-process $Y$, we can prove that strong ergodicity is equivalent to $\mathbb{L}^1$-exponential convergence, that is, there exist constants $1\leq\theta<\infty$ and $\gamma>0$ such that, for all $g\in\mathbb{L}^1(\beta)$ and $t\geq0$,
\begin{equation}
\label{3.1}
\|\widetilde{P}_{t}g-\beta(g)\|_{1}\leq \theta \|g-\beta(g)\|_{1}e^{-\gamma t},
\end{equation}
where $\|\cdot\|_{1}$ denotes the $\mathbb{L}^1(\beta)$-norm. In fact, define the projection operator $\Pi:\mathbb{L}^1(\beta)\rightarrow\mathbb{L}^1(\beta)$ by $\Pi g:=\beta(g)$. We introduce the following operator norms:
\begin{equation*}
\|\widetilde{P}_{t}-\Pi\|_{1\rightarrow1}:=\sup\limits_{\|g\|_{1}\leq1}\|(\widetilde{P}_{t}-\Pi)g\|_{1},
\end{equation*}
and
\begin{equation*}
\|\widetilde{P}_{t}-\Pi\|_{\infty\rightarrow\infty}:=\sup\limits_{\|g\|_{\infty}\leq1}\|(\widetilde{P}_{t}-\Pi)g\|_{\infty}.
\end{equation*}
By the standard duality $(\mathbb{L}^1(\beta))^{*}\cong\mathbb{L}^{\infty}(\beta)$, for any bounded linear operator $T:\mathbb{L}^1(\beta)\rightarrow\mathbb{L}^1(\beta)$, one has
\begin{equation*}
\|T\|_{1\rightarrow1}=\|T^{*}\|_{\infty\rightarrow\infty},
\end{equation*}
where $T^{*}$ denotes the adjoint operator of $T$. Since $(\widetilde{P}_{t})_{t\geq0}$ is reversible with respect to $\beta$, $\widetilde{P}_{t}$ is self-adjoint in $\mathbb{L}^2(\beta)$, i.e. $\widetilde{P}^{*}_{t}=\widetilde{P}_{t}$. Moreover, $\Pi^{*}=\Pi$. Hence,
\begin{equation*}
(\widetilde{P}_{t}-\Pi)^{*}=\widetilde{P}_{t}-\Pi.
\end{equation*}
So, for every $t\geq0$, we have
\begin{equation}
\label{3.22}
\|\widetilde{P}_{t}-\Pi\|_{1\rightarrow1}=\|\widetilde{P}_{t}-\Pi\|_{\infty\rightarrow\infty}.
\end{equation}
Fix $x>0$, by (\ref{2.222}), we know that the transition density $\widetilde{r}(s,x,\cdot)$ exists, so $\widetilde{P}_{t}(x,\cdot)\ll\beta$, and its Radon-Nikodym derivative is given by
\begin{equation*}
\frac{d\widetilde{P}_{t}(x,\cdot)}{d\beta}(y)=\widetilde{r}(s,x,y).
\end{equation*}
Consequently, $\widetilde{P}_{t}(x,\cdot)-\beta$ has density $\widetilde{r}(s,x,\cdot)-1$ with respect to $\beta$. By the integral representation of the total variation norm, one has
\begin{equation}
\label{3.23}
\left\|\widetilde{P}_{t}(x,\cdot)-\beta\right\|_{TV}=\frac{1}{2}\int_{0}^{\infty}|\widetilde{r}(s,x,y)-1|\beta(dy).
\end{equation}
On the other hand, for any $g\in\mathbb{L}^{\infty}(\beta)$,
\begin{equation*}
(\widetilde{P}_{t}-\Pi)g(x)=\int_{0}^{\infty}(\widetilde{r}(s,x,y)-1)g(y)\beta(dy).
\end{equation*}
Hence,
\begin{equation}
\label{3.24}
\sup\limits_{\|g\|_{\infty}\leq1}|(\widetilde{P}_{t}-\Pi)g(x)|=\int_{0}^{\infty}|\widetilde{r}(s,x,y)-1|\beta(dy),
\end{equation}
where the equality follows from the duality between $\mathbb{L}^1(\beta)$ and $\mathbb{L}^{\infty}(\beta)$. Combining (\ref{3.23}) and (\ref{3.24}), for fixed $x>0$,
\begin{equation}
\label{3.25}
\left\|\widetilde{P}_{t}(x,\cdot)-\beta\right\|_{TV}=\frac{1}{2}\sup\limits_{\|g\|_{\infty}\leq1}|(\widetilde{P}_{t}-\Pi)g(x)|.
\end{equation}
Taking the supremum over $x>0$ on both sides of (\ref{3.25}) yields
\begin{eqnarray}
\label{3.26}
\sup\limits_{x>0}\left\|\widetilde{P}_{t}(x,\cdot)-\beta\right\|_{TV}&=&\frac{1}{2}\sup\limits_{x>0}\sup\limits_{\|g\|_{\infty}\leq1}|(\widetilde{P}_{t}-\Pi)g(x)|\nonumber\\
                                                                     &=&\frac{1}{2}\sup\limits_{\|g\|_{\infty}\leq1}\|(\widetilde{P}_{t}-\Pi)g\|_{\infty}\nonumber\\
                                                                     &=&\frac{1}{2}\|\widetilde{P}_{t}-\Pi\|_{\infty\rightarrow\infty}.
\end{eqnarray}
Combining (\ref{3.22}) and (\ref{3.26}), we obtain
\begin{equation}
\label{3.27}
\sup\limits_{x>0}\left\|\widetilde{P}_{t}(x,\cdot)-\beta\right\|_{TV}=\frac{1}{2}\|\widetilde{P}_{t}-\Pi\|_{1\rightarrow1}.
\end{equation}
Thus, if the $Q$-process $Y$ is strongly ergodic, then by (\ref{3.27}), there exists two constants $k_{1},\gamma>0$ such that
\begin{equation*}
\|\widetilde{P}_{t}-\Pi\|_{1\rightarrow1}\leq2k_{1}e^{-\gamma t},~~~~\forall t\geq0.
\end{equation*}
For any $g\in\mathbb{L}^1(\beta)$, set $g_{0}:=g-\beta(g)$, then we have $\Pi g_{0}=0$ and
\begin{eqnarray*}
\|\widetilde{P}_{t}g-\beta(g)\|_{1}&=&\|(\widetilde{P}_{t}-\Pi)g_{0}\|_{1}\\
                                   &\leq&\|\widetilde{P}_{t}-\Pi\|_{1\rightarrow1}\|g_{0}\|_{1}\\
                                   &\leq&2k_{1}e^{-\gamma t}\|g-\beta(g)\|_{1}.
\end{eqnarray*}
Therefore, (\ref{3.1}) holds with $\theta=2k_{1}$. Conversely, if (\ref{3.1}) holds, then for any $g\in\mathbb{L}^1(\beta)$, $g_{0}=g-\beta(g)$, we have
\begin{eqnarray*}
\|\widetilde{P}_{t}g-\beta(g)\|_{1}&=&\|\widetilde{P}_{t}g_{0}\|_{1}\\
                                  &\leq&\theta e^{-\gamma t}\|g_{0}\|_{1}\\
                                  &\leq&2\theta e^{-\gamma t}\|g\|_{1}.
\end{eqnarray*}
Thus,
\begin{eqnarray*}
\|\widetilde{P}_{t}-\Pi\|_{1\rightarrow1}\leq2\theta e^{-\gamma t}.
\end{eqnarray*}
Substituting this into (\ref{3.27}), we obtain
\begin{equation*}
\sup\limits_{x>0}\left\|\widetilde{P}_{t}(x,\cdot)-\beta\right\|_{TV}\leq\theta e^{-\gamma t}.
\end{equation*}
Hence, the $Q$-process $Y$ is strongly ergodic.
\par
From Proposition \ref{prop 2.1}, we know that $\eta_1\in \mathbb{L}^1(m)$. So, for all $g\in\mathcal{B}_{1}(0,\infty)$, it is easy to see that $\frac{g}{\eta_1}\in\mathbb{L}^1(\beta)$.
Thus, for all probability measure $\mu$ on $(0,\infty)$, $g\in\mathcal{B}_{1}(0,\infty)$ and all $t\geq0$, we have
\begin{eqnarray*}
\mathbb{E}_{\mu}[g(X_{t})|T_{0}>t]&=&\frac{\int_{0}^{\infty}\mathbb{E}_{x}[g(X_{t}){\bf{1}}_{\{T_{0}>t\}}]\mu(dx)}{\int_{0}^{\infty}\mathbb{E}_{x}[{\bf{1}}_{\{T_{0}>t\}}]\mu(dx)}\\
                                   &=&\frac{\int_{0}^{\infty}\frac{e^{\lambda_{1}t}}{\eta_{1}(x)}\mathbb{E}_{x}[\eta_{1}(X_{t})\frac{g(X_{t})}{\eta_{1}(X_{t})}{\bf{1}}_{\{T_{0}>t\}}]\eta_{1}(x)\mu(dx)}{\int_{0}^{\infty}\frac{e^{\lambda_{1}t}}{\eta_{1}(x)}\mathbb{E}_{x}[\eta_{1}(X_{t})\frac{\bf{1}}{\eta_{1}(X_{t})}{\bf{1}}_{\{T_{0}>t\}}]\eta_{1}(x)\mu(dx)}\\
                                   &=&\frac{\int_{0}^{\infty}\widetilde{P}_{t}[\frac{g}{\eta_{1}}](x)\eta_{1}(x)\mu(dx)}{\int_{0}^{\infty}\widetilde{P}_{t}[\frac{\bf{1}}{\eta_{1}}](x)\eta_{1}(x)\mu(dx)}\\
                                   &=&\frac{(\eta_{1}\circ\mu)\widetilde{P}_{t}[\frac{g}{\eta_{1}}]}{(\eta_{1}\circ\mu)\widetilde{P}_{t}[\frac{\bf{1}}{\eta_{1}}]}.
\end{eqnarray*}
Here, $\bf{1}={\bf{1}}_{(0,\infty)}$.
\par
By (\ref{2.2}) and (\ref{2.4}), we have
\begin{equation*}
\beta(\frac{g}{\eta_{1}})=m(\eta_{1})\alpha(g).
\end{equation*}
So, for any $g\in\mathcal{B}_{1}(0,\infty)$, by (\ref{3.1}), there exists a constant $C'>0$ such that,
\begin{equation}
\label{3.2}
|(\eta_{1}\circ\mu)\widetilde{P}_{t}(\frac{g}{\eta_1})-m(\eta_{1})\alpha(g)|= |(\eta_{1}\circ\mu)\widetilde{P}_{t}(\frac{g}{\eta_{1}})-\beta(\frac{g}{\eta_{1}})|\leq C'e^{-\gamma t}
\end{equation}
and
\begin{equation}
\label{3.3}
|(\eta_{1}\circ\mu)\widetilde{P}_{t}(\frac{\bf{1}}{\eta_1})-m(\eta_{1})|\leq C'e^{-\gamma t}.
\end{equation}
Set $C_{1}:=\frac{C'}{m(\eta_{1})}$. Then, by (\ref{3.2}) and (\ref{3.3}), for any $t>\frac{\log C_{1}}{\gamma}$, one has
\begin{equation}
\label{3.4}
\frac{\alpha(g)-C_{1}e^{-\gamma t}}{1+C_{1}e^{-\gamma t}}\leq\mathbb{E}_{\mu}[g(X_{t})|T_{0}>t]\leq\frac{\alpha(g)+C_{1}e^{-\gamma t}}{1-C_{1}e^{-\gamma t}}.
\end{equation}
In addition, for any $t>\frac{\log C_{1}}{\gamma}$ and $g\in\mathcal {B}_{1}(0,\infty)$, we have
\begin{eqnarray*}
\frac{\alpha(g)+C_{1}e^{-\gamma t}}{1-C_{1}e^{-\gamma t}}&=&(\alpha(g)+C_{1}e^{-\gamma t})\left(1+\frac{C_{1}e^{-\gamma t}}{1-C_{1}e^{-\gamma t}}\right)\\
                                                 &\leq&\alpha(g)+C_{1}e^{-\gamma t}+(1+1)\frac{C_{1}e^{-\gamma t}}{1-C_{1}e^{-\gamma t}}\\
                                                 &=&\alpha(g)+C_{1}e^{-\gamma t}\left(1+\frac{2}{1-C_{1}e^{-\gamma t}}\right).
\end{eqnarray*}
We can pick any value smaller than 1, denoted by $d$, such that $C_{1}e^{-\gamma t}<d$. So
\begin{eqnarray*}
\frac{\alpha(g)+C_{1}e^{-\gamma t}}{1-C_{1}e^{-\gamma t}}&\leq&\alpha(g)+C_{2}e^{-\gamma t},
\end{eqnarray*}
where
\begin{eqnarray*}
C_{2}:=C_{1}\left(1+\frac{2}{1-d}\right).
\end{eqnarray*}
In a same way, for any $t>\frac{\log C_{1}}{\gamma}$ and $g\in\mathcal {B}_{1}(0,\infty)$, one can obtain
\begin{eqnarray*}
\frac{\alpha(g)-C_{1}e^{-\gamma t}}{1+C_{1}e^{-\gamma t}}&=&(\alpha(g)-C_{1}e^{-\gamma t})\left(1-\frac{C_{1}e^{-\gamma t}}{1+C_{1}e^{-\gamma t}}\right)\\
                                                 &\geq&\alpha(g)-C_{1}e^{-\gamma t}-(\alpha(g)+C_{1}e^{-\gamma t})\frac{C_{1}e^{-\gamma t}}{1+C_{1}e^{-\gamma t}}\\
                                                 &\geq&\alpha(g)-3C_{1}e^{-\gamma t}.
\end{eqnarray*}
Hence, for any $t>\frac{\log C_{1}}{\gamma}$, we obtain
\begin{eqnarray*}
\sup\limits_{g\in\mathcal {B}_{1}(0,\infty)}|\mathbb{E}_{\mu}[g(X_{t})|T_{0}>t]-\alpha(g)|\leq\max\{3C_{1}, C_{2}\}e^{-\gamma t}.
\end{eqnarray*}
This ends the proof of Theorem \ref{thm2.1}.

\section{Proof of Theorem \ref{thm2.2}}
\label{sect4}
The proof is similar to that of Theorem \ref{thm2.1}. First, by (\ref{1.7}), (\ref{2.2}) and (\ref{2.4}), one has
\begin{equation*}
\eta_1\circ\alpha(dx)=\beta(dx).
\end{equation*}
It can be seen that if $\|\frac{d(\eta_1\circ\mu)}{d\beta}-1\|_{2}=+\infty$, then the theorem is trivially true. So, we just need to consider initial distribution $\mu$ on $(0,\infty)$ such that $\|\frac{d(\eta_1\circ\mu)}{d\beta}-1\|_{2}<+\infty$.
\par
Since $\alpha(\psi^{2}/\eta_1)<+\infty$, for any measurable function $g$ on $(0,\infty)$ such that $|g|\leq\psi$, we have
\begin{equation*}
\|\frac{g}{\eta_1}\|_{2}\leq\left[m(\eta_{1})\alpha(\frac{\psi^2}{\eta_1})\right]^{\frac{1}{2}}<+\infty,
\end{equation*}
that is, $\frac{g}{\eta_1}\in\mathbb{L}^2(\beta)$. Moreover, note that the measure $\beta$ is a reversible measure for the semigroup $(\widetilde{P}_{t})_{t\geq0}$. Thus, for any initial distribution $\mu$ on $(0,\infty)$, we have
\begin{eqnarray*}
|\mu \widetilde{P}_{t}[\frac{g}{\eta_1}]-m(\eta_{1})\alpha(g)|&=&|\mu \widetilde{P}_{t}[\frac{g}{\eta_1}]-\beta(\frac{g}{\eta_1})|\\
                                                              &=&\left|\beta\left(\frac{d\mu}{d\beta}\widetilde{P}_{t}[\frac{g}{\eta_1}]-\frac{g}{\eta_1}\right)\right|\\
                                                              &=&\left|\beta\left(\frac{g}{\eta_1}\widetilde{P}_{t}\left(\frac{d\mu}{d\beta}\right)-\frac{g}{\eta_1}\right)\right|\\
                                                              &=&\left|\beta\left[\frac{g}{\eta_1}\left(\widetilde{P}_{t}\left(\frac{d\mu}{d\beta}-1\right)\right)\right]\right|.
\end{eqnarray*}
Further, from Proposition \ref{prop 2.2}, we know that the $Q$-process $Y$ is strongly ergodic. Then, according to \cite{C05}, $(\widetilde{P}_{t})_{t\geq0}$ has $\mathbb{L}^2$-exponential convergence, that is, there exists a positive constant $\gamma$ such that, for all $g\in\mathbb{L}^2(\beta)$ and $t\geq0$,
\begin{equation}
\label{4.1}
\|\widetilde{P}_{t}g-\beta(g)\|_{2}\leq e^{-\gamma t}\|g-\beta(g)\|_{2}.
\end{equation}
Hence, based on the Cauchy-Schwarz inequality and (\ref{4.1}), we obtain
\begin{eqnarray*}
\sup\limits_{|g|\leq\psi}|\mu \widetilde{P}_{t}[\frac{g}{\eta_1}]-m(\eta_{1})\alpha(g)|
                                                              &\leq&\left[\beta(\frac{\psi^2}{\eta^2_{1}})\right]^{\frac{1}{2}}\|\frac{d\mu}{d\beta}-1\|_{2}e^{-\gamma t}\\
                                                              &=&\left[m(\eta_{1})\alpha(\frac{\psi^2}{\eta_1})\right]^{\frac{1}{2}}\|\frac{d\mu}{d\beta}-1\|_{2}e^{-\gamma t}.
\end{eqnarray*}
Moreover, according to the proof of Theorem \ref{thm2.1}, it holds that
\begin{eqnarray*}
\mathbb{E}_{\mu}[g(X_{t})|T_{0}>t]=\frac{(\eta_{1}\circ\mu)\widetilde{P}_{t}[\frac{g}{\eta_1}]}{(\eta_{1}\circ\mu)\widetilde{P}_{t}[\frac{\bf{1}}{\eta_1}]}.
\end{eqnarray*}
Therefore, for any $t>\frac{\log[\frac{\alpha(\frac{\psi^2}{\eta_1})}{m(\eta_{1})}]^{\frac{1}{2}}\|\frac{d(\eta_{1}\circ\mu)}{d\beta}-1\|_{2}}{\gamma}$, we get
\begin{equation}
\label{4.2}
\frac{\alpha(g)-[\frac{\alpha(\frac{\psi^2}{\eta_1})}{m(\eta_{1})}]^{\frac{1}{2}}\|\frac{d(\eta_{1}\circ\mu)}{d\beta}-1\|_{2}e^{-\gamma t}}{1+[\frac{\alpha(\frac{\psi^2}{\eta_1})}{m(\eta_{1})}]^{\frac{1}{2}}\|\frac{d(\eta_{1}\circ\mu)}{d\beta}-1\|_{2}e^{-\gamma t}}
\leq\mathbb{E}_{\mu}[g(X_{t})|T_{0}>t]\leq\frac{\alpha(g)+[\frac{\alpha(\frac{\psi^2}{\eta_1})}{m(\eta_{1})}]^{\frac{1}{2}}\|\frac{d(\eta_{1}\circ\mu)}{d\beta}-1\|_{2}e^{-\gamma t}}{1-[\frac{\alpha(\frac{\psi^2}{\eta_1})}{m(\eta_{1})}]^{\frac{1}{2}}\|\frac{d(\eta_{1}\circ\mu)}{d\beta}-1\|_{2}e^{-\gamma t}}.
\end{equation}
We pick any value smaller than 1, denoted by $c$, such that $[\frac{\alpha(\frac{\psi^2}{\eta_1})}{m(\eta_{1})}]^{\frac{1}{2}}\|\frac{d(\eta_{1}\circ\mu)}{d\beta}-1\|_{2}e^{-\gamma t}<c$.
Thus, by an argument similar to that of Theorem \ref{thm2.1}, for any $t>\frac{\log[\frac{\alpha(\frac{\psi^2}{\eta_1})}{m(\eta_{1})}]^{\frac{1}{2}}\|\frac{d(\eta_{1}\circ\mu)}{d\beta}-1\|_{2}}{\gamma}$ and $|g|\leq\psi$, we have
\begin{eqnarray*}
\frac{\alpha(g)+[\frac{\alpha(\frac{\psi^2}{\eta_1})}{m(\eta_{1})}]^{\frac{1}{2}}\|\frac{d(\eta_{1}\circ\mu)}{d\beta}-1\|_{2}e^{-\gamma t}}{1-[\frac{\alpha(\frac{\psi^2}{\eta_1})}{m(\eta_{1})}]^{\frac{1}{2}}\|\frac{d(\eta_{1}\circ\mu)}{d\beta}-1\|_{2}e^{-\gamma t}}&\leq&\alpha(g)+D_{1}[\frac{\alpha(\frac{\psi^2}{\eta_1})}{m(\eta_{1})}]^{\frac{1}{2}}\|\frac{d(\eta_{1}\circ\mu)}{d\beta}-1\|_{2}e^{-\gamma t},
\end{eqnarray*}
where
\begin{eqnarray*}
D_{1}:=\left(1+\frac{1+\alpha(\psi)}{1-c}\right).
\end{eqnarray*}
Similarly, for any $t>\frac{\log[\frac{\alpha(\frac{\psi^2}{\eta_1})}{m(\eta_{1})}]^{\frac{1}{2}}\|\frac{d(\eta_{1}\circ\mu)}{d\beta}-1\|_{2}}{\gamma}$ and $|g|\leq\psi$, one has
\begin{eqnarray*}
\frac{\alpha(g)-[\frac{\alpha(\frac{\psi^2}{\eta_1})}{m(\eta_{1})}]^{\frac{1}{2}}\|\frac{d(\eta_{1}\circ\mu)}{d\beta}-1\|_{2}e^{-\gamma t}}{1+[\frac{\alpha(\frac{\psi^2}{\eta_1})}{m(\eta_{1})}]^{\frac{1}{2}}\|\frac{d(\eta_{1}\circ\mu)}{d\beta}-1\|_{2}e^{-\gamma t}}
\geq\alpha(g)-D_{2}[\frac{\alpha(\frac{\psi^2}{\eta_1})}{m(\eta_{1})}]^{\frac{1}{2}}\|\frac{d(\eta_{1}\circ\mu)}{d\beta}-1\|_{2}e^{-\gamma t},
\end{eqnarray*}
where
\begin{eqnarray*}
D_{2}:=2+\alpha(\psi).
\end{eqnarray*}
So by (\ref{4.2}), for any $t>\frac{\log[\frac{\alpha(\frac{\psi^2}{\eta_1})}{m(\eta_{1})}]^{\frac{1}{2}}\|\frac{d(\eta_{1}\circ\mu)}{d\beta}-1\|_{2}}{\gamma}$, we obtain
\begin{eqnarray*}
\sup\limits_{|g|\leq\psi}|\mathbb{E}_{\mu}[g(X_{t})|T_{0}>t]-\alpha(g)|\leq \max\{D_{1}, D_{2}\}[\frac{\alpha(\frac{\psi^2}{\eta_1})}{m(\eta_{1})}]^{\frac{1}{2}}\|\frac{d(\eta_{1}\circ\mu)}{d\beta}-1\|_{2}e^{-\gamma t}.
\end{eqnarray*}
\par
Let $$\varphi_{t}(\mu):=\mathbb{P}_{\mu}(X_{t}\in\cdot|T_{0}>t).$$
As mentioned above, we get that there exists $t_{\mu}\geq0$ such that, for any $t\geq t_{\mu}$,
\begin{equation*}
[\frac{\alpha(\frac{\psi^2}{\eta_1})}{m(\eta_{1})}]^{\frac{1}{2}}\|\frac{d(\eta_{1}\circ\varphi_{t}(\mu))}{d\beta}-1\|_{2}e^{-\gamma t}<c.
\end{equation*}
Also, based on (\ref{1.7}) and (\ref{2.5}), for any probability measure $\mu$ on $(0,\infty)$, any measurable function $g$ on $(0,\infty)$ and $t\geq0$, it is easy to show that
\begin{eqnarray*}
\eta_{1}\circ\varphi_{t}(\mu)(g)&=&\frac{\varphi_{t}(\mu)(\eta_{1}g)}{\varphi_{t}(\mu)(\eta_{1})}\\
                                &=&\frac{e^{\lambda_{1}t}\mu P_{t}(\eta_{1}g)}{\mu(\eta_{1})}\\
                                &=&(\eta_{1}\circ\mu)\widetilde{P}_{t}(g),
\end{eqnarray*}
that is,
\begin{equation}
\label{4.3}
\eta_{1}\circ\varphi_{t}(\mu)=(\eta_{1}\circ\mu)\widetilde{P}_{t}.
\end{equation}
Therefore, by (\ref{4.1}) and (\ref{4.3}), for any $t\geq t_{\mu}$, we obtain
\begin{eqnarray*}
\sup\limits_{|g|\leq\psi}|\mathbb{E}_{\mu}[g(X_{t})|T_{0}>t]-\alpha(g)|&\leq& \max\{D_{1}, D_{2}\}[\frac{\alpha(\frac{\psi^2}{\eta_1})}{m(\eta_{1})}]^{\frac{1}{2}}\|\frac{d(\eta_{1}\circ\varphi_{t_{\mu}}(\mu))}{d\beta}-1\|_{2}e^{-\gamma (t-t_{\mu})}\\
&\leq&\max\{D_{1}, D_{2}\}[\frac{\alpha(\frac{\psi^2}{\eta_1})}{m(\eta_{1})}]^{\frac{1}{2}}\|\frac{d(\eta_{1}\circ\mu)}{d\beta}-1\|_{2}e^{-\gamma t}.
\end{eqnarray*}
This ends the proof of Theorem \ref{thm2.2}.

\section{Proof of Proposition \ref{pro1.333}}
\label{sect5}
In this section, we are devoted to giving the proof of Proposition \ref{pro1.333}. Firstly, note that $\eta_{1}(x)$ is bounded on $(0,\infty)$. Also, for all $x\in(0,\infty)$, all bounded measurable function $g$ on $(0,\infty)$ and $0<p<1$, we have
\begin{eqnarray*}
&&\mathbb{E}_{x}\left[g(X_{pt})|T_{0}>t\right]-\beta(g)\\
&=&\frac{e^{\lambda_{1} t}\mathbb{E}_{x}\left[g(X_{pt}){\bf{1}}_{\{T_{0}>t\}}\right]}{\eta_{1}(x)}\cdot\frac{\eta_{1}(x)}{e^{\lambda_{1} t}\mathbb{P}_{x}(T_{0}>t)}-\frac{e^{\lambda_{1} t}\mathbb{E}_{x}\left[g(X_{pt}){\bf{1}}_{\{T_{0}>t\}}\right]}{\eta_{1}(x)}\\
&+&\frac{e^{\lambda_{1} t}\mathbb{E}_{x}\left[g(X_{pt}){\bf{1}}_{\{T_{0}>t\}}\right]}{\eta_{1}(x)}-\frac{e^{\lambda_{1} pt}\mathbb{E}_{x}\left[g(X_{pt})\eta_{1}(X_{pt}){\bf{1}}_{\{T_{0}>pt\}}\right]}{\eta_{1}(x)}\\
&+&\frac{e^{\lambda_{1} pt}\mathbb{E}_{x}\left[g(X_{pt})\eta_{1}(X_{pt}){\bf{1}}_{\{T_{0}>pt\}}\right]}{\eta_{1}(x)}-\beta(g).
\end{eqnarray*}
Moreover, according to \cite[Theorem 2.1]{CV17b}, there exists a constant $a_{1}>0$ such that, for all $t\geq0$,
\begin{equation*}
\left|e^{\lambda_{1} t}\mathbb{P}_{x}(T_{0}>t)-\eta_{1}(x)\right|\leq a_{1}e^{\lambda_{1} t}\mathbb{P}_{x}(T_{0}>t)e^{-\gamma t}.
\end{equation*}
So, we get
\begin{equation}
\label{5.1}
(1-a_{1}e^{-\gamma t})e^{\lambda_{1} t}\mathbb{P}_{x}(T_{0}>t)\leq\eta_{1}(x)\leq (1+a_{1}e^{-\gamma t})e^{\lambda_{1} t}\mathbb{P}_{x}(T_{0}>t).
\end{equation}
Thus, for $t$ large enough, we have
\begin{equation}
\label{5.2}
\frac{\eta_{1}(x)}{1+a_{1}e^{-\gamma t}}\leq e^{\lambda_{1} t}\mathbb{P}_{x}(T_{0}>t)\leq \frac{\eta_{1}(x)}{1-a_{1}e^{-\gamma t}}.
\end{equation}
From (\ref{5.2}), for all $t>\frac{\log a_{1}}{\gamma}$, we get
\begin{eqnarray*}
&&\left|\frac{e^{\lambda_{1} t}\mathbb{E}_{x}\left[g(X_{pt}){\bf{1}}_{\{T_{0}>t\}}\right]}{\eta_{1}(x)}\cdot\frac{\eta_{1}(x)}{e^{\lambda_{1} t}\mathbb{P}_{x}(T_{0}>t)}-\frac{e^{\lambda_{1} t}\mathbb{E}_{x}\left[g(X_{pt}){\bf{1}}_{\{T_{0}>t\}}\right]}{\eta_{1}(x)}\right|\\
&\leq&\|g\|_{\infty}\left|\frac{e^{\lambda_{1} t}\mathbb{P}_{x}(T_{0}>t)}{\eta_{1}(x)}\right|\left|\frac{\eta_{1}(x)}{e^{\lambda_{1} t}\mathbb{P}_{x}(T_{0}>t)}-1\right|\\
&\leq&\|g\|_{\infty}\cdot\frac{1}{1-a_{1}e^{-\gamma t}}\cdot a_{1}e^{-\gamma t}\\
&=&\frac{a_{1}e^{-\gamma t}}{1-a_{1}e^{-\gamma t}}\|g\|_{\infty}.
\end{eqnarray*}
Furthermore, for all $t>\frac{\log a_{1}}{\gamma(1-p)}$, we can deduce that
\begin{equation*}
\left|e^{\lambda_{1}(1-p)t}\mathbb{P}_{X_{pt}}(T_{0}>(1-p)t)-\eta_{1}(X_{pt})\right|\leq\frac{a_{1}e^{-\gamma(1-p)t}}{1-a_{1}e^{-\gamma(1-p)t}}\eta_{1}(X_{pt}).
\end{equation*}
Thus, using the Markov property, we obtain
\begin{eqnarray*}
&&\left|\frac{e^{\lambda_{1} t}\mathbb{E}_{x}\left[g(X_{pt}){\bf{1}}_{\{T_{0}>t\}}\right]}{\eta_{1}(x)}-\frac{e^{\lambda_{1} pt}\mathbb{E}_{x}\left[g(X_{pt})\eta_{1}(X_{pt}){\bf{1}}_{\{T_{0}>pt\}}\right]}{\eta_{1}(x)}\right|\\
&=&\left|\frac{e^{\lambda_{1} pt}\mathbb{E}_{x}\left[g(X_{pt}){\bf{1}}_{\{T_{0}>pt\}}\cdot e^{\lambda_{1}(1-p)t} \mathbb{E}_{X_{pt}}({\bf{1}}_{\{T_{0}>(1-p)t\}})\right]}{\eta_{1}(x)}-\frac{e^{\lambda_{1} pt}\mathbb{E}_{x}\left[g(X_{pt})\eta_{1}(X_{pt}){\bf{1}}_{\{T_{0}>pt\}}\right]}{\eta_{1}(x)}\right|\\
&=&\left|\frac{e^{\lambda_{1} pt}\mathbb{E}_{x}\left[g(X_{pt}){\bf{1}}_{\{T_{0}>pt\}}\cdot e^{\lambda_{1}(1-p)t}\mathbb{P}_{X_{pt}}(T_{0}>(1-p)t)\right]}{\eta_{1}(x)}-\frac{e^{\lambda_{1} pt}\mathbb{E}_{x}\left[g(X_{pt})\eta_{1}(X_{pt}){\bf{1}}_{\{T_{0}>pt\}}\right]}{\eta_{1}(x)}\right|\\
&\leq&\|g\|_{\infty}\frac{e^{\lambda_{1} pt}\mathbb{E}_{x}\left[{\bf{1}}_{\{T_{0}>pt\}}\left|e^{\lambda_{1}(1-p)t}\mathbb{P}_{X_{pt}}(T_{0}>(1-p)t)-\eta_{1}(X_{pt})\right|\right]}{\eta_{1}(x)}\\
&\leq&\|g\|_{\infty}\frac{a_{1}e^{-\gamma(1-p)t}}{1-a_{1}e^{-\gamma(1-p)t}}\frac{e^{\lambda_{1} pt}\mathbb{E}_{x}\left[\eta_{1}(X_{pt}){\bf{1}}_{\{T_{0}>pt\}}\right]}{\eta_{1}(x)}\\
&=&\|g\|_{\infty}\frac{a_{1}e^{-\gamma(1-p)t}}{1-a_{1}e^{-\gamma(1-p)t}}\left(\frac{e^{\lambda_{1} pt}}{\eta_{1}(x)}\cdot P_{pt}\eta_{1}(x)\right)\\
&=&\|g\|_{\infty}\frac{a_{1}e^{-\gamma(1-p)t}}{1-a_{1}e^{-\gamma(1-p)t}}\left(\frac{e^{\lambda_{1} pt}}{\eta_{1}(x)}\cdot e^{-\lambda_{1} pt}\eta_{1}(x)\right)\\
&=&\frac{a_{1}e^{-\gamma(1-p)t}}{1-a_{1}e^{-\gamma(1-p)t}}\|g\|_{\infty}.
\end{eqnarray*}
According to Proposition \ref{prop 2.2}, the $Q$-process $Y$ is strongly ergodic, consequently there exist two positive constants $C'', \gamma$ such that, for any $x\in(0,\infty)$ and $t\geq0$,
\begin{equation}
\label{5.333}
\left\|\mathbb{Q}_{x}(Y_t\in \cdot)-\beta\right\|_{TV}\leq C''e^{-\gamma t}.
\end{equation}
So, from (\ref{2.5}) and (\ref{5.333}), for any $x\in(0,\infty)$ and $t\geq0$, we have
\begin{eqnarray*}
&&\left|\frac{e^{\lambda_{1} pt}\mathbb{E}_{x}\left[g(X_{pt})\eta_{1}(X_{pt}){\bf{1}}_{\{T_{0}>pt\}}\right]}{\eta_{1}(x)}-\beta(g)\right|\\
&=&\left|\widetilde{P}_{pt}g(x)-\beta(g)\right|\\
&\leq&C''e^{-\gamma pt}\|g\|_{\infty}.
\end{eqnarray*}
Therefore, we get
\begin{eqnarray*}
&&\left|\mathbb{E}_{x}\left[g(X_{pt})|T_{0}>t\right]-\beta(g)\right|\\
&\leq&\frac{a_{1}e^{-\gamma t}}{1-a_{1}e^{-\gamma t}}\|g\|_{\infty}+\frac{a_{1}e^{-\gamma(1-p)t}}{1-a_{1}e^{-\gamma(1-p)t}}\|g\|_{\infty}+C''e^{-\gamma pt}\|g\|_{\infty}\\
&\leq&\frac{2a_{1}e^{-\gamma(1-p)t}}{1-a_{1}e^{-\gamma(1-p)t}}\|g\|_{\infty}+C''e^{-\gamma pt}\|g\|_{\infty}.
\end{eqnarray*}
So, there exists a constant $a_{2}>0$ such that, for all $t\geq0$,
\begin{equation}
\label{5.3}
\left|\mathbb{E}_{x}\left[g(X_{pt})|T_{0}>t\right]-\beta(g)\right|\leq a_{2}\|g\|_{\infty}(e^{-\gamma(1-p)t}+e^{-\gamma pt}).
\end{equation}
Set $s=pt$. From (\ref{5.3}), we have
\begin{eqnarray*}
\left|\mathbb{E}_{x}\left(\frac{1}{t}\int_{0}^{t}g(X_s)ds|T>t\right)-\beta(g)\right|
&=&\left|\mathbb{E}_{x}\left(\int_{0}^{1}g(X_{pt})dp|T_{0}>t\right)-\beta(g)\right|\\
&=&\left|\int_{0}^{1}[\mathbb{E}_{x}\left(g(X_{pt})|T_{0}>t\right)-\beta(g)]dp\right|\\
&\leq&\int_{0}^{1}\left|\mathbb{E}_{x}\left(g(X_{pt})|T_{0}>t\right)-\beta(g)\right|dp\\
&\leq&a_{2}\|g\|_{\infty}\int_{0}^{1}(e^{-\gamma(1-p)t}+e^{-\gamma pt})dp\\
&=&a_{2}\|g\|_{\infty}\frac{2(1-e^{-\gamma t})}{\gamma t}.
\end{eqnarray*}
This ends the proof of Proposition \ref{pro1.333}.

\section{Two examples}
\label{sect6}
In this section, we apply our main results to the following two processes arising in population dynamics.\\

{\bf{Example 1.}}
The first example is the logistic Feller diffusion process on $(0,\infty)$ killed at 0, which is defined
as the solution of the SDE
\begin{equation*}
dZ_t=\sqrt{\sigma Z_{t}}dB_t+\left(rZ_{t}-kZ^{2}_{t}\right)dt,~~~~~~~~~~Z_{0}=z>0,
\end{equation*}
where $\sigma,r,k$ are positive constants and $(B_{t})_{t\geq0}$ is the standard one-dimensional Brownian motion.
The logistic Feller diffusion process is a classic biological model and has strong biological background and applications. See \cite{CCLMMS09} for more information.
\par
Define $X_t:=2\sqrt{Z_{t}/\sigma}$. By It\^{o}'s formula, we get
\begin{eqnarray*}
dX_t=dB_t-\left(\frac{1}{2X_{t}}-\frac{rX_{t}}{2}+\frac{k\sigma X_{t}^{3}}{8}\right)dt,~~~~~X_0=x=2\sqrt{z/\sigma}>0.
\end{eqnarray*}
According to \cite[Lemma 3.3]{M14}, the point 0 is an exit boundary and the point $\infty$ is an entrance boundary for the process $X$. Therefore, Theorem \ref{thm2.1} holds for the process $X$, and then Theorem \ref{thm2.1} holds for the process $Z$. Further, if there exists a function $\psi:(0,\infty)\rightarrow[1,+\infty)$ such that $\alpha(\psi)<+\infty$ and $\alpha(\psi^{2}/\eta_1)<+\infty$, where $\alpha$ is the unique quasi-stationary distribution of the process $X$, then for any probability measure $\mu$ on $(0,\infty)$, Theorem \ref{thm2.2} holds for the process $X$, so that for the process $Z$.\\

{\bf{Example 2.}} The second example is a process coming up naturally as a scaling limit of branching diffusions in a random
environment \cite{BH11,HM11}. It is defined as the solution of the SDE
\begin{equation}
\label{6.1}
dX_t=\left(rX_{t}-kX^{2}_{t}\right)dt+\gamma X_{t}dW_t+\sqrt{\sigma X_{t}}dB_t,~~~~~X_0>0,
\end{equation}
where $\gamma,\sigma,r,k$ are positive constants, $(W_{t})_{t\geq0}$ and $(B_{t})_{t\geq0}$ are independent Brownian motions.
\par
Because the quadratic variation of the process $X$ is
\begin{equation*}
d[X]_{t}=\left(\sigma X_{t}+\gamma^{2} X^{2}_{t}\right)dt,
\end{equation*}
there exists a Brownian motion $(U_{t})_{t\geq0}$ such that
\begin{equation}
\label{6.3}
dX_t=\left(rX_{t}-kX^{2}_{t}\right)dt+\sqrt{\sigma X_{t}+\gamma^{2} X^{2}_{t}}dU_{t}.
\end{equation}
According to \cite[Proposition 3.8]{HK19}, the point 0 is an exit boundary and the point $\infty$ is an entrance boundary for the diffusion
given by (\ref{6.3}). Therefore, Theorems \ref{thm2.1} and \ref{thm2.2} hold for the diffusion given by (\ref{6.1}).

%\section*{Funding}
\section*{Acknowledgements}
This work was supported by the National Natural Science Foundation of China (Grant No. 12001184), Outstanding Youth Project of Education Department of Hunan Province (Grant No. 19B307) and Natural Science Foundation of Hunan Province (Grant No. 2019JJ50289).

%\section*{Author Contributions}
%All authors have contributed equally to this work.

%\section*{Conflict of Interest}
%The authors declare no conflict of interest.

%\section*{Data Availability Statement}
%Data availability is not applicable to this article.

%\section*{Clinical Trial Number}
%Clinical trial number: not applicable.

\end{document}